\documentclass[a4paper,12pt,reqno]{amsart}

\usepackage{a4wide}
\usepackage[T1]{fontenc}
\usepackage{enumitem}
\usepackage{url}
\usepackage{xcolor}
\usepackage{mathtools}
\usepackage{dsfont}

\newtheorem{lemma}{Lemma}
\newtheorem{proposition}[lemma]{Proposition}
\newtheorem{theorem}[lemma]{Theorem}

\theoremstyle{remark}
\newtheorem{remark}{Remark}

\newcommand{\Ex}{\mathbb{E}}
\newcommand{\Va}{\mathbb{V}}
\newcommand{\IP}{\mathbb{P}}
\newcommand{\Poi}{\mathsf{Poi}}
\newcommand{\Bin}{\mathsf{Bin}}
\newcommand{\ER}{Erd\H{o}s--R\'enyi\ }
\newcommand{\ssq}{\subseteq}
\newcommand{\ga}{\alpha}
\newcommand{\gep}{\varepsilon}
\newcommand{\gd}{\delta}
\newcommand{\gl}{\lambda}

\newcommand{\IN}{\mathbb{N}}
\newcommand{\IR}{\mathbb{R}}

\title{On the intersection of pairs of trees}

\author[M. B\'ona]{Mikl\'os B\'ona}
\address{Department of Mathematics, University of Florida, Gainesville, FL 32611-8105, USA; bona@ufl.edu}

\author[F. Burghart]{Fabian Burghart}
\address{Institute of Discrete Mathematics, TU Graz, Steyrergasse 30, 8010 Graz, Austria, fabian.burghart@tugraz.at}

\author[S. Wagner]{Stephan Wagner}
\address{Institute of Discrete Mathematics, TU Graz, Steyrergasse 30, 8010 Graz, Austria, and
Department of Mathematics, Uppsala University, Box 480, 751 06 Uppsala, Sweden, stephan.wagner@tugraz.at}

\thanks{F.B. was supported by the Austrian Science Fund (FWF) [10.55776/F1002]. S.W. was supported by the Swedish research council (Vetenskapsr{\aa}det), grant 2022-04030.}

\begin{document}

\begin{abstract} We consider the number of common edges in two independent random spanning trees of a graph $G$. For complete graphs $K_n$, we give a new proof of the fact, originally obtained by Moon, that the distribution converges to a Poisson distribution with expected value $2$. This is applied to show a Poisson limit law for the number of common edges in two independent random spanning trees of an Erd\H{o}s--R\'enyi random graph $G(n,p)$ for constant~$p$, as well as a central limit theorem in the case where $p\to 0$ and $p\geq n^{-2/3+\gep}$. We also use the same method to prove an analogous result for complete multipartite graphs.
\end{abstract}

\maketitle

\section{Introduction}
For two given combinatorial structures $S_1$ and $S_2$ taken over the same underlying set and selected uniformly at random, there is intrinsic interest in studying the distribution of $|S_1\cap S_2|$. An early paper devoted to such questions is \cite{aldous} by  Aldous. Trees are frequent objects in 
this line of research. For instance, in \cite{bryant}, Bryant, McKenzie, and Steel  proved  that the maximum size of a common subtree for two independent copies of a uniformly random binary tree with $n$ leaves is likely to be of order
$O(n^{1/2})$.  Pittel \cite{pittel} proved the bound $O(n^{1/2})$ for the maximum size of a subtree common to two independent copies of a random rooted tree, namely the terminal tree of a critical Galton--Watson branching process
 conditioned on the total number of leaves being $n$. In \cite{bcp}, the first author, Costin and Pittel proved lower and upper bounds for the likely size of the largest twin subtrees (two subtrees that have the same count of vertices by their degrees) of the uniformly random rooted Cayley tree.

In this paper, we consider the following natural question. Let us independently select two spanning trees of the same graph $G$, both uniformly at random. What can be said about the number of edges that the two have in common? For the complete graph $K_n$, the answer is given (implicitly) by Moon in \cite{moon2}: the number of pairs of trees that have exactly $m$ edges in common is
\[\sum_{j=0}^{n-m-1} (-1)^j \binom{m+j}{m} n^{2(n-m-j-2)} \frac{(n-1)!}{(n-m-j-1)!} \sum_{k=0}^{m+j} \binom{n-k}{m+j-k} \frac{n^k}{k!}.\]
We get their proportion by dividing by the total number of pairs (which is $n^{2(n-2)}$):
\[\sum_{j=0}^{n-m-1} (-1)^j \binom{m+j}{m} n^{-2m-2j} \frac{(n-1)!}{(n-m-j-1)!} \sum_{k=0}^{m+j} \binom{n-k}{m+j-k} \frac{n^k}{k!}.\]
Note that
\[n^{-2m-2j} \frac{(n-1)!}{(n-m-j-1)!} \binom{n-k}{m+j-k} \frac{n^k}{k!} 
= \frac{(n-1)^{\underline{m+j}}(n-k)^{\underline{m+j-k}}n^{-2m-2j+k}}{(m+j-k)!k!},\]
where $x^{\underline{\ell}} = x(x-1)\cdots (x-\ell+1)$ denotes a falling factorial. For fixed $\ell$, $x^{\ell}$ is both an upper bound for and asymptotically equal to $x^{\underline{\ell}}$ as $x \to \infty$. So for fixed $m$, $j$ and $k$, we have
\begin{align*}n^{-2m-2j} \frac{(n-1)!}{(n-m-j-1)!} \binom{n-k}{m+j-k} \frac{n^k}{k!} &\sim 
\frac{(n-1)^{m+j}(n-k)^{m+j-k}n^{-2m-2j+k}}{(m+j-k)!k!} \\
&\sim \frac{1}{(m+j-k)!k!}
\end{align*}
as well as
\begin{align}n^{-2m-2j} \frac{(n-1)!}{(n-m-j-1)!} \binom{n-k}{m+j-k} \frac{n^k}{k!} &\leq
\frac{(n-1)^{m+j}(n-k)^{m+j-k}n^{-2m-2j+k}}{(m+j-k)!k!} \nonumber \\
&\leq \frac{1}{(m+j-k)!k!}.\label{eq:aux-upper-bound}
\end{align}
Since
\begin{equation}\label{eq:bound}
 \sum_{j \geq 0} \binom{m+j}{m} \sum_{k=0}^{m+j} \frac{1}{(m+j-k)!k!} = \sum_{j \geq 0} \binom{m+j}{m} \frac{2^{m+j}}{(m+j)!} = \sum_{j \geq 0} \frac{2^{m+j}}{m!j!} = \frac{2^m}{m!} e^2
\end{equation}
converges, we can apply dominated convergence to find that the proportion of pairs of trees that intersect in precisely $m$ edges tends to
\begin{equation}\label{eq:limit}
 \sum_{j \geq 0} (-1)^j \binom{m+j}{m} \sum_{k=0}^{m+j} \frac{1}{(m+j-k)!k!} = \sum_{j \geq 0} (-1)^j \binom{m+j}{m} \frac{2^{m+j}}{(m+j)!} = \frac{2^m}{m!} e^{-2}.
\end{equation}
Very recently, Fuchs and Steel~\cite{fuchs} also proved this result  using a different method. Their
motivation came from medical bioinformatics~\cite{khayatian}. We remark that the number of common edges in two random spanning trees was also recently studied for complete graphs with random weights~\cite{makowiec}.

So the limiting distribution of the number of common edges is a Poisson distribution with expected value $2$. Heuristically, this can be explained as follows: each edge has the same probability $\frac{2}{n}$ to belong to a spanning tree, thus a probability of $\frac{4}{n^2}$ to belong to two independently chosen spanning trees. Since there are $\binom{n}{2}$ edges that are (as we will see) ``almost independent'', the law of rare events suggests convergence to a Poisson distribution.

In the following section, we provide an alternative proof of the Poisson limit for complete graphs by proving convergence of moments. In Section~\ref{sec:Gnp}, using a conditioning argument, this is then applied to pairs of random spanning trees in the Erd\H{o}s--R\'enyi random graph $G(n,p)$ for constant $p$, yielding a Poisson-limit with parameter $2/p$. For $p\to 0$ sufficiently slowly, we also prove a central limit theorem.

Finally, an analogous result is shown for complete multipartite graphs in Section~\ref{sec:multipartite}. The paper concludes with a brief discussion of the analogous problem for arbitrary graphs.

\section{Complete graphs}

Consider two spanning trees on the same vertex set $[n] = \{1,2,\ldots,n\}$, both selected uniformly at random among all $n^{n-2}$ possibilities. We are interested in the number of edges these two trees have in common.

We need the following known counting formula (see \cite{cameron,moon}). 

\begin{lemma}\label{lem:given_components}
For a given forest $F$ on $[n]$ with components of sizes $n_1,n_2,\ldots,n_k$, the number of spanning trees containing $F$ as a subgraph is $(n_1n_2 \cdots n_k) n^{k-2}$.
\end{lemma}
In particular, the probability that a fixed edge is contained in a random spanning tree on $[n]$ is
\begin{equation} \label{average} \frac{2 n^{n-3}}{n^{n-2}} = \frac{2}{n}.\end{equation}
Moreover, the probability that two adjacent edges are both contained in a random spanning tree on $[n]$ is
\[\frac{3 n^{n-4}}{n^{n-2}} = \frac{3}{n^2},\]
and the probability that two vertex-disjoint edges are both contained is
\[\frac{4n^{n-4}}{n^{n-2}} = \frac{4}{n^2},\]
which means that the two events are independent. Now let us consider a pair of two spanning trees with vertex set $[n]$, both chosen uniformly at random, and let $X_n$ be the number of common edges. For every possible edge $e$, we let $Z_e$ be the indicator random variable that is $1$ if both trees contain $e$, and $0$ otherwise. We can then write
\begin{equation}\label{eq:sum_of_ind}
X_n = \sum_e Z_e.
\end{equation}
We consider the moments of $X_n$ now.

\begin{proposition}\label{prop:moments}
For every fixed positive integer $r$, the $r$-th moment of $X_n$ is asymptotically equal to
\[e^{-2} \sum_{k \geq 0} \frac{2^k k^r}{k!} + O(n^{-1}),\]
i.e., it converges to the $r$-th moment of a Poisson random variable with expected value $2$.
\end{proposition}

\begin{proof}
By~\eqref{eq:sum_of_ind},
\[\Ex(X_n^r) = \sum_{e_1,e_2,\ldots,e_r}  \Ex(Z_{e_1}Z_{e_2} \cdots Z_{e_r}),\]
the sum being over all ordered $r$-tuples of (not necessarily distinct) edges. The summands can be simplified to terms of the form
\[\Ex(Z_{e_1}^{a_1} Z_{e_2}^{a_2} \cdots Z_{e_s}^{a_s}),\]
where the $e_i$ are distinct. Since the $Z_{e_i}$ are indicator random variables, this is equal to
\begin{equation}\label{eq:red_term}
\Ex(Z_{e_1}Z_{e_2}\cdots Z_{e_s}).
\end{equation}
The sum of the coefficients of all terms that reduce to~\eqref{eq:red_term} is
\[\sum_{\substack{a_1,a_2,\ldots,a_s \geq 1 \\ a_1 + a_2 + \cdots + a_s = r}} \binom{r}{a_1,a_2,\ldots,a_r}
= r! [t^r] (e^t-1)^s.\]
If $e_1,e_2,\ldots,e_s$ are disjoint edges, then by Lemma~\ref{lem:given_components}, we have
\[\Ex(Z_{e_1}Z_{e_2}\cdots Z_{e_s}) = \Big(\frac{2^s n^{n-s-2}}{n^{n-2}}\Big)^2 = 2^{2s} n^{-2s}.\]
Now consider an arbitrary set of $s$ distinct edges $e_1,e_2,\ldots,e_s$. If some of these edges form a cycle, then clearly $\Ex(Z_{e_1}Z_{e_2}\cdots Z_{e_s}) = 0$ as spanning trees cannot contain cycles. Otherwise, they induce a forest with $n-s$ components on the vertex set $[n]$. If $m_1,m_2,\ldots,m_{n-s}$ are its component sizes, then by Lemma~\ref{lem:given_components} there are
\[(m_1m_2\cdots m_{n-s})n^{n-s-2}\]
spanning trees that contain all the edges $e_1,e_2,\ldots,e_s$. By the elementary inequality $m \leq 2^{m-1}$ that is valid for all positive integers $m$, we have
\[(m_1m_2\cdots m_{n-s})n^{n-s-2} \leq 2^{m_1+m_2+\cdots+m_{n-s}-n+s} n^{n-s-2} = 2^s n^{n-s-2}.\]
It follows that
\[\Ex(Z_{e_1}Z_{e_2}\cdots Z_{e_s}) = 
\Big(\frac{(m_1m_2\cdots m_{n-s})n^{n-s-2}}{n^{n-2}}\Big)^2
\leq \Big(\frac{2^s n^{n-s-2}}{n^{n-2}}\Big)^2 = 2^{2s} n^{-2s},\]
thus $\Ex(Z_{e_1}Z_{e_2}\cdots Z_{e_s}) = O(n^{-2s})$ in all cases.
Of the $\binom{n(n-1)/2}{s}$ different sets of $s$ edges, the number of sets of disjoint edges is
\[\binom{n}{2s} (2s-1)!! = \frac{n(n-1)\cdots (n-2s+1)}{2^s s!},\]
since there are $\binom{n}{2s}$ choices for the vertices and $(2s-1)!!$ ways to pair them. Note that $\binom{n(n-1)/2}{s}$ and $\binom{n}{2s} (2s-1)!!$ are both $\frac{n^{2s}}{2^s s!} + O(n^{2s-1})$, implying that the number of sets of $s$ edges that are not all disjoint is $O(n^{2s-1})$. So it follows that
\begin{align*}
\Ex(X_n^r) &= \sum_{s=1}^r \sum_{\{e_1,e_2,\ldots,e_s\}} r! [t^r] (e^t-1)^s \Ex(Z_{e_1}Z_{e_2}\cdots Z_{e_s}) \\
&= \sum_{s=1}^r r! [t^r] (e^t-1)^s \Big( \binom{n}{2s} (2s-1)!! \frac{2^{2s}}{n^{2s}} + O \Big( n^{2s-1} \cdot n^{-2s} \Big) \Big) \\
&= \sum_{s=1}^r r! [t^r] (e^t-1)^s \Big( \frac{2^s}{s!} + O(n^{-1}) \Big).
\end{align*}
Thus the $r$-th moment of $X_n$ converges to
\begin{align*}
\sum_{s=1}^r r! [t^r] (e^t-1)^s \frac{2^s}{s!} &= r! [t^r] \sum_{s = 0}^{\infty} \frac{2^s (e^t-1)^s}{s!} = r! [t^r] e^{2(e^t-1)} \\
&= r! e^{-2} [t^r] \sum_{k \geq 0} \frac{(2e^t)^k}{k!} = e^{-2} \sum_{k \geq 0} \frac{2^k k^r}{k!},
\end{align*}
completing the proof.
\end{proof}

\begin{remark}
Carrying out the precise calculations for the first two moments, one finds that the expected value of $X_n$ is
\[\Ex(X_n) = \frac{2(n-1)}{n},\]
while the variance is
\[\Va(X_n) = \frac{(n-1)(n-2)(2n-3)}{n^3}.\]
\end{remark}

Since the Poisson distribution is characterized by its moments, Proposition~\ref{prop:moments} immediately implies that $X_n$ converges in distribution. Using the Chen--Stein method, specifically in the formulation of \cite{AGG89}, we can obtain an additional bound on the total variation distance. Recall here that the total variation distance of two probability measures $P$ and $Q$ on the same measure space $(\Omega,\mathcal{F})$ is defined as
    \[d_{TV}(P,Q) = \sup_{A \in \mathcal{F}} |P(A) - Q(A)|.\]   

\begin{theorem}\label{thm:labelled1}
The number of common edges in a pair of uniform random spanning trees converges in distribution to a Poisson distribution with expected value $2$. Thus the probability that the number of common edges is equal to $m$ converges to $\frac{2^m}{e^2 m!}$. Moreover, all moments converge, and we have
    \[
        d_{TV}(X_n,\Poi(2)) \leq \frac{27}{n}.
    \]
\end{theorem}

\begin{proof}
    It only remains to prove the bound on the total variation distance. Continuing the notation from the proof of Proposition~\ref{prop:moments}, we consider the same indicator random variables $Z_e$. For every edge $e$, define $B_e$ to be the set of edges intersecting $e$ (including $e$ itself). Next we show that $Z_e$ is independent of $\{Z_{e'}:e'\in B_e^c\}$. Consider any subset $M$ of $B_e^c$, and let $Z_M = \prod_{e' \in M} Z_{e'}$ be the indicator random variable of the event that all edges of $M$ are contained in both spanning trees. If some of the edges of $M$ form a cycle, then $Z_M$ is always $0$, thus $\Ex(Z_eZ_M) = 0 = \Ex(Z_e)\Ex(Z_M)$. Otherwise, let $m_1,m_2,\ldots,m_r$ be the component sizes of the forest induced on the vertex set $[n]$ by the edges in $M$ (including components consisting of a single vertex, two of which are the ends of $e$). Thus it follows from Lemma~\ref{lem:given_components} that 
    \[\Ex(Z_eZ_M) = \Big( \frac{(2m_1m_2 \cdots m_r)n^{r-3}}{n^{n-2}} \Big)^2 = \Big( \frac{2}{n} \Big)^2 \Big( \frac{(m_1m_2 \cdots m_r)n^{r-2}}{n^{n-2}} \Big)^2 = \Ex(Z_e)\Ex(Z_M).\]
    This means that $Z_e$ is independent of $Z_M$ for all $M \subseteq B_e^c$, and thus independent of $\{Z_{e'}:e'\in B_e^c\}$.
    
Now Theorem~1 of \cite{AGG89} gives the following estimate for the total variation distance between $X_n$ and a Poisson distribution with the same expected value $2 - \frac{2}{n}$:
    \begin{align*}
        d_{TV}(X_n,\Poi(2-2/n)) 
        &\leq \left( \sum_{e} \sum_{e'\in B_e} \Ex[Z_e]\Ex[Z_{e'}] + \sum_{e} \sum_{e\neq e'\in B_e} \Ex[Z_eZ_{e'}]\right)\\
        &= \left( \binom{n}{2} (2n-3) \frac{16}{n^4} + \binom{n}{2} (2n-4) \frac{9}{n^4} \right) \leq \frac{25}{n}.
    \end{align*}
    Combined with an estimate on the total variation distance between Poisson distributions of different parameters, such as $d_{TV}(\Poi(2-2/n),\Poi(2))\leq 1-e^{-2/n} \leq 2/n$ (cf. equation (2.2) in \cite{AJ05}), this yields 
    \[
        d_{TV}(X_n,\Poi(2)) \leq \frac{27}{n}.\qedhere
    \]
\end{proof}

\begin{remark}
For $k \geq 3$ independent trees, the intersection is empty with high probability: since every edge occurs in a random spanning tree of $K_n$ with probability $\frac{2}{n}$, the probability that there is a common edge is at most
\[\binom{n}{2} \Big( \frac{2}{n} \Big)^k\]
by the union bound, which goes to $0$ for every $k \geq 3$.

One can be slightly more precise: recall that the probability that two given edges are both contained in a random spanning tree of $K_n$ is either $\frac{3}{n^2}$ or $\frac{4}{n^2}$ (depending on whether they are adjacent or not). So we can apply the Bonferroni inequality
\[P \Big( \bigcup_{1 \leq i \leq \binom{n}{2}} A_i \Big) \geq \sum_{1 \leq i \leq \binom{n}{2}} P(A_i) - \sum_{1 \leq i < j \leq \binom{n}{2}} P(A_i \cap A_j),\]
where $A_i$ stands for the event that the $i$-th edge is contained in all $k$ trees, to estimate the probability that there is a common edge from below as well: it is at least
\[\binom{n}{2} \Big( \frac{2}{n} \Big)^k - \frac{n(n-1)(n-2)}{2} \Big( \frac{3}{n^2} \Big)^k - \frac{n(n-1)(n-2)(n-3)}{8} \Big( \frac{4}{n^2} \Big)^k.\]
This means that the probability that $k$ random spanning trees of $K_n$ (where $k \geq 3$) have nonempty intersection is $\binom{n}{2} \big( \frac{2}{n} \big)^k - O\big(n^{4-2k}\big)$.
It is also easy to see that conditioned on the intersection being nonempty, the probability that it consists of a single edge tends to $1$.
\end{remark}

\section{\ER random graphs}\label{sec:Gnp}

By an application of Bayes' law, we can generalize the Poisson convergence in Theorem~\ref{thm:labelled1} to the \ER random graph $G(n,p)$ for constant $p$:
\begin{theorem}\label{thm:random}
 Let $G\sim G(n,p)$ for constant $p>0$ and denote by $X$ the number of joint edges between two independent, uniformly chosen spanning trees of $G$, where we declare $X=0$ in the event that $G$ is disconnected. Then $X\to \Poi(2/p)$ in distribution, that is, for every $m\geq 0$,
 \begin{equation}\label{eq:Poi2p}
  \IP(X=m) \to e^{-2/p} \frac{(2/p)^m}{m!}
 \end{equation}
 as $n\to\infty$.
\end{theorem}
\begin{proof}
Let $T_1,T_2$ be two independent uniform spanning trees of $K_n$. If they have $m$ edges in common, then their union has $2(n-1)-m$ edges by the inclusion-exclusion principle. Thus $\IP(T_1,T_2\ssq G\mid |T_1\cap T_2|=m) = p^{2(n-1)-m}$, and we obtain
\[
\IP(T_1,T_2\ssq G, |T_1\cap T_2|=m) = p^{2n-2-m} \IP (|T_1\cap T_2|=m).
\]
Consequently,
\begin{align}\label{eq:proofconstp}
 \frac{\IP(T_1,T_2\ssq G)}{p^{2n-2}}
 &= \sum_{m=0}^{n-1} \frac{\IP(T_1,T_2\ssq G, |T_1\cap T_2|=m)}{p^{2n-2}}\notag\\
 &= \sum_{m=0}^{n-1} p^{-m} \IP(|T_1\cap T_2|=m)\\
 &\to \sum_{m=0}^\infty  p^{-m}e^{-2} \frac{2^m}{m!} = e^{-2+2/p}\notag,
\end{align}
where the exchange of sum and limit is permitted by the dominated convergence theorem, recycling the argument of Section 1. Conditioned on the event that $G$ is connected and $T_1,T_2 \subseteq G$, $T_1$ and $T_2$ are still independent, and both are uniformly distributed on the set of spanning trees. Therefore, we have
\begin{align*}
 &\IP(X=m\mid G\text{ connected}) 
        = \IP( |T_1\cap T_2|=m \mid T_1,T_2\ssq G) \\
 &\qquad\qquad= \frac{\IP( T_1,T_2\ssq G \mid |T_1\cap T_2|=m) \cdot \IP(|T_1\cap T_2|=m)}{\IP(T_1,T_2\ssq G)}\\
 &\qquad\qquad\sim \frac{p^{2n-2-m}\cdot e^{-2}\frac{2^m}{m!}}{p^{2n-2}e^{-2+2/p}} = e^{-2/p}\frac{(2/p)^m}{m!}.
\end{align*}
Since $G$ is connected with high probability, we obtain \eqref{eq:Poi2p}. 
\end{proof}

We also consider the case that $p\to0$. The proof technique of Theorem~\ref{thm:labelled1} that was based on the Chen--Stein method does not seem to apply in this setting, so we will use a more direct approach, which has the additional benefit that it provides a local central limit theorem as long as $p$ does not tend to $0$ too quickly. 
\begin{theorem}\label{thm:pto0}
 Let $G\sim G(n,p)$ for $p=p(n)\to 0$ with $p\geq n^{-2/3+\gep}$ for some fixed $\gep > 0$, and let $X$ be as in Theorem~\ref{thm:random}. Set $\gl_n=2/p(n)$. If $\ga_n \to \ga\in\IR$ as $n\to\infty$ such that $m_n=\gl_n+\ga_n\sqrt{\gl_n}\in \IN_0$, then 
 \begin{equation}\label{eq:LCLT}
  \sqrt{\gl_n} \IP(X=m_n) \to \frac{1}{\sqrt{2\pi}} e^{-\frac{\ga^2}{2}}
 \end{equation}
 as $n\to\infty$, where the convergence is uniform in $\ga$ on compact sets. In particular, we have that $\gl_n^{-1/2}(X-\gl_n) \to N(0,1)$ in distribution.
\end{theorem}

The proof rests on a sequence of technical lemmas whose goal is to estimate $\IP(X=m)$. 
As a start, we analyze the asymptotic behavior of $\IP(|T_1\cap T_2|=m)$ in greater detail. Let us set
 \begin{align*}
  r_m(n)
  &:=\frac{m!}{2^m e^{-2}} \IP(|T_1\cap T_2|=m) \\
  &=\frac{m!}{2^m e^{-2}} \sum_{j=0}^{n-m-1} (-1)^j \binom{m+j}{m} n^{-2m-2j} \frac{(n-1)!}{(n-m-j-1)!} \sum_{k=0}^{m+j} \binom{n-k}{m+j-k} \frac{n^k}{k!}
 \end{align*}
for $0 \leq m < n$ (and $r_m(n) = 0$ otherwise). We note that $r_m(n)\to 1$ pointwise for fixed $m$ as $n\to\infty$ by \eqref{eq:limit}, and $r_m(n)\leq e^4$ for all $m,n$ by \eqref{eq:bound}. Finally, let us write
 \begin{equation}\label{eq:defSnmj}
  S_{n,m}(j) := \frac{m!}{2^m e^{-2}}(-1)^j \binom{m+j}{m} n^{-2m-2j} \frac{(n-1)!}{(n-m-j-1)!} \sum_{k=0}^{m+j} \binom{n-k}{m+j-k} \frac{n^k}{k!},
 \end{equation}
 so that $r_m(n) = \sum_{j=0}^{n-m-1} S_{n,m}(j)$. Now observe that 
 \[
  n^{-m-j} \frac{(n-1)!}{(n-m-j-1)!} = \left(1-\frac{1}{n}\right)\cdots \left(1-\frac{m+j}{n}\right)
 \]
 and similarly
 \[
  n^{-m-j}\binom{n-k}{m+j-k} \frac{n^k}{k!} = \frac{1}{(m+j)!}\binom{m+j}{k} \left(1-\frac{k}{n}\right) \cdots \left( 1-\frac{m+j-1}{n}\right).
 \]
 Plugging in both of these identities into \eqref{eq:defSnmj} yields
 \begin{align*}\label{eq:Taylor}
  S_{n,m}(j)
  &= \frac{m!}{2^m e^{-2}} (-1)^j \frac{1}{(m+j)!} \binom{m+j}{m} \prod_{i=1}^{m+j} \left(1-\frac{i}{n}\right) \sum_{k=0}^{m+j} \binom{m+j}{k} \prod_{\ell=k}^{m+j-1} \left(1-\frac{\ell}{n}\right)\notag\\
  &= \frac{e^2(-2)^j}{j!} \left(1-\frac{m+j}{n}\right)^{-1} \cdot \prod_{i=1}^{m+j} \left(1-\frac{i}{n}\right)^2 \cdot \frac{1}{2^{m+j}} \sum_{k=0}^{m+j} \binom{m+j}{k} \prod_{s=1}^{k-1} \left(1-\frac{s}{n}\right)^{-1}.
 \end{align*}

\begin{lemma}\label{lem:Snm-lem}
Fix $\gd \in (0,\frac23)$, and suppose that $m \leq n^{2/3-\gd}$ and $j \leq \sqrt{m}$. Then we have
\[  S_{n,m}(j) = \frac{e^2(-2)^j}{j!} \exp \Big( {-}\frac{7m^2}{8n} + O(n^{-\gd}) \Big),\]
where the error term is uniform in $m$ and $j$.
\end{lemma}

\begin{proof}
Note that
\[1-\frac{m+j}{n} = \exp \Big( O \Big( \frac{m}{n} \Big) \Big)\]  and, making use of the Taylor expansion of $\ln(1-x)$,
\begin{align*}
\prod_{i=1}^{m+j} \left(1-\frac{i}{n}\right)^2
&= \exp \Big( {-} \sum_{i=1}^{m+j} \Big( \frac{2i}{n} + O \Big( \frac{i^2}{n^2} \Big) \Big) \Big) \\
&= \exp \Big( {-} \frac{(m+j)(m+j+1)}{n}
+ O \Big( \frac{m^3}{n^2} \Big) \Big) \\
&= \exp \Big( {-} \frac{m^2}{n} + O \Big( \frac{m^{3/2}}{n} + \frac{m^3}{n^2} \Big) \Big).
\end{align*}
By a similar calculation,
\[\prod_{s=1}^{k-1} \left(1-\frac{s}{n}\right)^{-1} = \exp\Big(
\frac{k^2}{2n} + O \Big( \frac{m^{3/2}}{n} + \frac{m^3}{n^2} \Big) \Big)
\]
for $k \leq m+j$. Thus, by the assumptions on $m$, 
\[S_{n,m}(j) = \frac{e^2(-2)^j}{j!} 
\exp \Big( {-} \frac{m^2}{n} + O(n^{-\gd}) \Big) \cdot 
\frac{1}{2^{m+j}} \sum_{k=0}^{m+j} \binom{m+j}{k} \exp \Big( \frac{k^2}{2n} + O(n^{-\gd}) \Big),
\]
where the error term in the sum is uniform in $k$. It remains to deal with the inner sum. We are done if we can show that
\[\Sigma := \frac{1}{2^{m+j}} \sum_{k=0}^{m+j} \binom{m+j}{k}
\exp \Big( \frac{k^2}{2n} \Big) = \exp \Big( \frac{m^2}{8n} + O(n^{-\gd}) \Big).
\]
This is trivial if $m < n^{(1-\gd)/2}$, since then $\frac{m^2}{8n} = O(n^{-\gd})$ and $\frac{k^2}{2n} = O(n^{-\gd})$ for $k \leq m+j$, so we can assume that $m \geq n^{(1-\gd)/2}$.
Now observe that the left side can be interpreted as $\Ex(e^{Y^2/(2n)})$, where $Y$ is a $\Bin(m+j,\frac12)$-distributed random variable. By a standard Chernoff bound \cite[Theorem 2.1]{JLR00}, we have
\[\IP\Big(\Big|Y - \frac{m+j}{2}\Big| > t\Big) = 2\IP\Big(Y - \frac{m+j}{2} < t\Big) \leq 2 \exp \Big( {-} \frac{t^2}{m+j} \Big).\]
We apply this with $t = m^{3/2} n^{\gd - 1/2}$: then
\[\frac{t^2}{m+j} \sim \frac{m^3 n^{2\gd-1}}{m} = m^2n^{2\gd-1} \geq n^{\gd}
\]
goes to infinity, and it does so faster than
\[\frac{k^2}{2n} \leq \frac{(m+j)^2}{2n} \sim m^2n^{-1}.\]
This means that the contribution of terms with $|k - \frac{m+j}{2}| > t$ is negligible (it tends to $0$ faster than any power of $n$). If, on the other hand, $|k - \frac{m+j}{2}| \leq t$, then we have
\[\frac{k^2}{2n} = - \frac{(m+j)^2}{8n} + \frac{k(m+j)}{2n} + \frac{1}{2n} \Big( k - \frac{m+j}{2} \Big)^2,\]
and the final term can be bounded as follows:
\[\frac{1}{2n} \Big( k - \frac{m+j}{2} \Big)^2 \leq \frac{t^2}{2n} = \frac{m^3}{2n^{2-2\gd}} \leq \frac{n^{2-3\gd}}{2n^{2-2\gd}} = \frac{1}{2n^{\gd}}.\]
So we conclude that
\begin{align*}
\Sigma &= 
\exp \Big( {-} \frac{(m+j)^2}{8n} + O(n^{-\gd}) \Big)
\frac{1}{2^{m+j}} \sum_{k=0}^{m+j} \binom{m+j}{k} \exp \Big( \frac{k(m+j)}{2n} \Big) \\
&= \exp \Big( {-} \frac{(m+j)^2}{8n} + O(n^{-\gd}) \Big)
\Big( \frac{1 + e^{(m+j)/(2n)}}{2} \Big)^{m+j} \\
&= \exp \Big( {-} \frac{(m+j)^2}{8n} + O(n^{-\gd}) \Big) \cdot \exp \Big( \frac{(m+j)^2}{4n} + O \Big( \frac{(m+j)^3}{n^2} \Big) \Big) \\
&= \exp \Big( \frac{(m+j)^2}{8n} + O(n^{-\gd}) \Big) \\
&= \exp \Big( \frac{m^2}{8n} + O(n^{-\gd}) \Big)
\end{align*}
with a uniform error term, which completes the proof of the lemma.
\end{proof}

\begin{lemma}\label{lem:rm-lem}
Fix $\gd \in (0,\frac23)$, and suppose that $m \leq n^{2/3-\gd}$. Then we have
\[  r_m(n) = \exp \Big( {-}\frac{7m^2}{8n} + O(n^{-\gd}) \Big) + O\big(e^{-\sqrt{m}}\big),\]
uniformly in $m$.    
\end{lemma}

\begin{proof}
    We split the sum $r_m(n) = \sum_{j=0}^{n-m-1} S_{n,m}(j)$ into the part with $j \leq \sqrt{m}$ and the rest. For $j > \sqrt{m}$, we use the inequality
    \begin{align*}
|S_{n,m}(j)| &= \frac{m!}{2^m e^{-2}} \binom{m+j}{m} n^{-2m-2j} \frac{(n-1)!}{(n-m-j-1)!} \sum_{k=0}^{m+j} \binom{n-k}{m+j-k} \frac{n^k}{k!} \\
    &\leq \frac{m!}{2^m e^{-2}} \binom{m+j}{m} \sum_{k=0}^{m+j} \frac{1}{(m+j-k)!k!} = \frac{e^22^j}{j!},
    \end{align*}
which follows from~\eqref{eq:aux-upper-bound}. Since $\frac{2^j}{j!} = O(e^{-j})$, we have
\[\sum_{j > \sqrt{m}} S_{n,m}(j) = O \big( e^{-\sqrt{m}} \big).\]
Now, since the error term in Lemma~\ref{lem:Snm-lem} is uniform in $j$, we obtain
\begin{align*}
r_m(n) &= \sum_{j \leq m} S_{n,m}(j) + O \big( e^{-\sqrt{m}} \big) = \sum_{j \leq m} \frac{e^2(-2)^j}{j!} \exp \Big( {-}\frac{7m^2}{8n} + O(n^{-\gd}) \Big) + O \big( e^{-\sqrt{m}} \big) \\
&= \exp \Big( {-}\frac{7m^2}{8n} + O(n^{-\gd}) \Big) + O \big( e^{-\sqrt{m}} \big).
\end{align*} 
\end{proof}

We are now ready to prove Theorem~\ref{thm:pto0}.

\begin{proof}[Proof of Theorem~\ref{thm:pto0}]
With slight abuse of notation, we drop the dependence of $\gl_n,m_n,\ga_n$ on $n$ and only write $\gl,m,\ga$ instead.
    We apply Bayes' law as the proof of Theorem~\ref{thm:random}. Since $G(n,p)$ is connected with high probability for $p\geq (1+\gep)\frac{\ln n}{n}$, we have
 \begin{align*}
  \IP(X=m) 
  &\sim \IP(X=m\mid G\text{ connected })
   = \IP(|T_1\cap T_2|=m \mid T_1,T_2\ssq G)\\
  &=\frac{\IP(T_1,T_2\ssq G \mid |T_1\cap T_2|=m)\IP(|T_1\cap T_2|=m)}{\IP(T_1,T_2\ssq G)}\\
  &= \frac{p^{2n-2-m}\IP(|T_1\cap T_2|=m)}{\IP(T_1,T_2\ssq G)}.
 \end{align*}
Recall that $r_m(n) = \frac{m!}{2^m e^{-2}} \IP(|T_1\cap T_2|=m)$, or equivalently $\IP(|T_1\cap T_2|=m) = \frac{2^m e^{-2}r_m(n)}{m!}$. It follows that
\[\IP(T_1,T_2\ssq G) = \sum_{h \geq 0} p^{2n-2-h} \IP(|T_1\cap T_2|=h) = p^{2n-2}e^{-2} \sum_{h \geq 0} \frac{2^h r_h(n)}{h!p^h},\]
and thus
\begin{equation}\label{eq:PX=m}
\IP(X=m) \sim \frac{\frac{2^m r_m(n)}{m!p^m}}{\sum_{h \geq 0} \frac{2^h r_h(n)}{h!p^h}} = \frac{\frac{\gl^m r_m(n)}{m!}}{\sum_{h \geq 0} \frac{\gl^h r_h(n)}{h!}}.
\end{equation}
 Recall that we are assuming that $p \geq n^{-2/3+\gep}$. Moreover, we know that $r_m(n)\leq e^4$ for all $m,n$ by \eqref{eq:bound}, and Lemma~\ref{lem:rm-lem} with $\gd = \frac{\gep}{2}$ yields
\begin{equation}\label{eq:rmn-asy}
    r_m(n) = \exp \Big( {-}\frac{7m^2}{8n} + O(n^{-\gep/2}) \Big) + O\big(e^{-\sqrt{m}}\big)
\end{equation}
for $m \leq n^{2/3-\gep/2}$, uniformly in $m$.
Let $M$ be a $\Poi(\gl)$-distributed random variable. Then
\[e^{-\gl} \sum_{h \geq 0} \frac{\gl^h r_h(n)}{h!} = \Ex(r_{M}(n)).\]
By the Chernoff bounds for the Poisson distribution (cf. \cite[Remark~2.6]{JLR00}), we have
 \begin{equation}\label{eq:chernoff-poi}
   \IP(|M-\gl|\geq t) \leq 2 \exp \Big( {-} \frac{t^2}{2(\gl+t/3)} \Big)
 \end{equation}
for $t>0$. We apply this with $t = \frac{\gl}{2}$ (noting that $p\to 0$ and hence $\gl\to\infty$) and use the constant upper bound $r_m(n)\leq e^4$ to obtain
\[e^{-\gl} \sum_{h \geq 0} \frac{\gl^h r_h(n)}{h!} = e^{-\gl} \sum_{|h-\gl| \leq \gl/2} \frac{\gl^h r_h(n)}{h!} + O(e^{-3\gl/28}).\]
The error term at the end of~\eqref{eq:rmn-asy} is $O(e^{-\sqrt{\gl/2}})$ for $|m -\gl| \leq \frac{\gl}{2}$. So it follows that
\[e^{-\gl} \sum_{h \geq 0} \frac{\gl^h r_h(n)}{h!} =
e^{-\gl} \Big( 1 + O(n^{-\gep/2}) \Big) \sum_{|h-\gl| \leq \gl/2} \frac{\gl^h}{h!} \exp \Big( {-}\frac{7h^2}{8n} \Big) + O\big(e^{-\sqrt{\gl/2}}\big).\]
So it suffices to consider the sum
\begin{equation}\label{eq:chernoff-hsum}
 e^{-\gl} \sum_{|h-\gl| \leq \gl/2} \frac{\gl^h}{h!} \exp \Big( {-}\frac{7h^2}{8n} \Big).
\end{equation}
We distinguish two cases depending on how large $\gl$ is (equivalently, how small $p$ is).
Assume first that $\gl < n^{(1-\gep)/2}$. Then we have $\frac{7h^2}{8n} \leq \frac{63\gl^2}{32n} = O(n^{-\gep})$ for $|h-\gl| \leq \gl/2$ as well as
$\frac{7\gl^2}{8n} = O(n^{-\gep})$, so
\[e^{-\gl} \sum_{|h-\gl| \leq \gl/2} \frac{\gl^h}{h!} \exp \Big( {-}\frac{7h^2}{8n} \Big) = \exp \Big( {-}\frac{7\gl^2}{8n} + O(n^{-\gep}) \Big) + O\big(e^{-\sqrt{\gl/2}}\big),\]
implying finally that
\begin{equation}\label{eq:sum-asy}
e^{-\gl} \sum_{h \geq 0} \frac{\gl^h r_h(n)}{h!} =  \exp \Big( {-}\frac{7\gl^2}{8n} + O(n^{-\gep/2}) \Big) + O\big(e^{-\sqrt{\gl/2}}\big).
\end{equation}
If, on the other hand, $\gl \geq n^{(1-\gep)/2}$, then
we proceed as in the proof of Lemma~\ref{lem:Snm-lem}. We apply the Chernoff bound~\eqref{eq:chernoff-poi} once again, this time with $t = \gl^{3/2}n^{\gep-1/2}$. We may assume here that $\gep < \frac13$, so that
\[t = \gl \cdot \gl^{1/2} \cdot n^{\gep - 1/2} \leq \gl \big(2n^{2/3-\gep}\big)^{1/2} \cdot n^{\gep - 1/2} \leq \sqrt{2} \gl n^{\gep/2 - 1/6} \leq \frac{3\gl}{2}.\]
It then follows that
\[\IP(|M-\gl|\geq t) \leq 2 e^{-t^2/(3\gl)}.\]
Now
\[\frac{t^2}{3\gl} = \frac{\gl^2 n^{2\gep-1}}{3} \geq \frac{n^{\gep}}{3}\]
goes to infinity, and it does so faster than $\frac{7\gl^2}{8n}$, so all indices $h$ in~\eqref{eq:chernoff-hsum} with $|h-\gl| > t$ only contribute another error term. Finally, for $|h-\gl| \leq t$, we have
\begin{align*}
\frac{7h^2}{8n} &= - \frac{7\gl^2}{8n} + \frac{7\gl h}{4n} + \frac{7(h-\gl)^2}{8n} \\
&= - \frac{7\gl^2}{8n} + \frac{7\gl h}{4n} + O \Big( \frac{t^2}{n} \Big) \\
&= - \frac{7\gl^2}{8n} + \frac{7\gl h}{4n} + O \Big( \gl^3 n^{2\gep-2} \Big) \\
&= - \frac{7\gl^2}{8n} + \frac{7\gl h}{4n} + O \big( n^{-\gep} \big).
\end{align*}
Thus
\[\sum_{|h-\gl| \leq t} \frac{\gl^h}{h!} \exp \Big( {-}\frac{7h^2}{8n} \Big)
= \exp \Big( \frac{7\gl^2}{8n} + O \big( n^{-\gep} \big) \Big) \sum_{|h-\gl| \leq t} \frac{1}{h!} \big( \gl e^{-7\gl/(4n)} \big)^h.\]
The final sum can be interpreted in terms of a random variable $\tilde{M}$ that follows a Poisson distribution with mean $\tilde{\gl} = \gl e^{-7\gl/(4n)}  = \gl + O (\frac{\gl^2}{n})$: we have
\begin{align*}
    \sum_{|h-\gl| \leq t} \frac{1}{h!} \big( \gl e^{-7\gl/(4n)} \big)^h = \exp \Big( \gl e^{-7\gl/(4n)} \Big) \IP \Big( |\tilde{M} - \gl| \leq t \Big).
\end{align*}
Now
\[\IP \Big( |\tilde{M} - \gl| > t \Big) \leq \IP \Big( |\tilde{M} - \tilde{\gl}| > t - |\gl - \tilde{\gl}|\Big),\]
and another application of the Chernoff bound~\eqref{eq:chernoff-poi} with $\gl$ replaced by $\tilde{\gl}$ and $t$ replaced by $\tilde{t} = t - |\gl - \tilde{\gl}|$ (note here that $|\gl - \tilde{\gl}| = O(\frac{\gl^2}{n}) = O(tn^{-\gep})$) yields
\[\IP \Big( |\tilde{M} - \gl| > t \Big) = O \Big( \exp \Big( {-} \frac{n^{\gep}}{3} \Big) \Big),\]
thus
\[\sum_{|h-\gl| \leq t} \frac{\gl^h}{h!} \exp \Big( {-}\frac{7h^2}{8n} \Big)
= \exp \Big( \frac{7\gl^2}{8n} + \gl e^{-7\gl/(4n)} + O \big( n^{-\gep} \big) \Big) = \exp \Big( \gl - \frac{7\gl^2}{8n} + O \big( n^{-\gep} \big) \Big).\]
So putting everything together, we find once again that~\eqref{eq:sum-asy} holds. Now, under the conditions of Theorem~\ref{thm:pto0}, we can apply \eqref{eq:rmn-asy} to $m = \gl + \ga \sqrt{\gl}$. Together with~\eqref{eq:PX=m} and~\eqref{eq:sum-asy}, it gives us
\begin{align*}
\IP(X = m) &\sim \frac{\frac{\gl^mr_m(n)}{m!}}{\sum_{h \geq 0} 
\frac{\gl^hr_h(n)}{h!}} \sim \frac{\gl^m}{m!} \exp \Big( {-} \gl + \frac{7\gl^2}{8n} - \frac{7m^2}{8n} \Big) \\
&= e^{-\gl} \frac{\gl^m}{m!} \exp \Big( {-} \gl + O \Big( \frac{\gl^{3/2}}{n} \Big) \Big) \sim e^{-\gl} \frac{\gl^m}{m!},    
\end{align*}
with error terms that are uniform in $\ga$ as required. The rest is the standard approximation of the Poisson distribution by a Gaussian distribution for large $\gl$: by Stirling's formula, we have
\[\IP(X = m) \sim \exp \Big( {-}\gl + m \ln \gl - m \ln m + m - \frac12 \ln(2\pi m) + O(m^{-1}) + O \Big( \frac{\gl^{3/2}}{n} \Big) \Big),\]
and the expression inside the exponential simplifies to 
\[- \frac{\ga^2}{2} - \frac12 \ln(2\pi \gl) + O (\gl^{-1/2}) + O \Big( \frac{\gl^{3/2}}{n} \Big),\]
completing the proof of the local central limit theorem \eqref{eq:LCLT}. The convergence in distribution then follows by observing that 
 \[
  \IP\left(a\leq \frac{X-\gl}{\sqrt{\gl}}\leq b\right) 
  = \sum_{m=\lceil\gl+a\sqrt{\gl}\rceil}^{\lfloor \gl+b\sqrt{\gl}\rfloor} \IP(X=m) 
  \to \int_a^b \frac{1}{\sqrt{2\pi}} e^{-\ga^2/2} \,\mathrm{d}\ga
 \]
 due to the uniformity for all $[a,b]\ssq \IR$.
\end{proof}

\begin{remark}
It is very likely that the exponent $\frac23$ in Theorem~\ref{thm:pto0} is not best possible. With a more detailed asymtotic analysis, it might even be possible to prove the theorem for all $p$ above the threshold for connectedness.    
\end{remark}

\section{Complete multipartite graphs}\label{sec:multipartite}

We now prove a result similar to Theorem~\ref{thm:labelled1} for complete multipartite graphs. The key to our proof of Proposition~\ref{prop:moments} and thus to Theorem~\ref{thm:labelled1} was the fact that disjoint edges are independent. Asymptotically, this will also be the case for the complete multipartite graphs that we consider. We will make use of a recent formula due to Li, Chen and Yan \cite{lcy} for the number of spanning trees in complete multipartite graphs that contain a fixed forest. 

\begin{lemma}\label{lem:lcy-formula}
Let $K_{n_1,n_2,\ldots,n_d}$ be a complete multipartite graph with vertex partition $V_1 \cup V_2 \cup \cdots \cup V_d$ (such that $|V_i| = n_i$), and let $F$ be a spanning forest with components $T_1,T_2,\ldots,T_k$. For $i \in \{1,2,\ldots,d\}$ and $j \in \{1,2,\ldots,k\}$, let $n_{ij}$ be the number of vertices of $T_j$ that lie in $V_i$. Now set
\[\ga_j = \sum_{i=1}^d (n-n_i)n_{ij}.\]
Then the number of spanning trees that contain $F$ is
\[\frac{\prod_{j=1}^k \ga_j}{(d-1)^{d-2} \prod_{i=1}^d (n-n_i)^2} \sum_{T \in \mathcal{T}(K_d)} \prod_{v_pv_q \in E(T)} (n-n_p)(n-n_q) \big( 1 - (d-1)a_{pq} \big),\]
where $a_{pq} = \sum_{j=1}^k \frac{n_{pj}n_{qj}}{\ga_j}$, and the sum goes over the set $\mathcal{T}(K_d)$ of spanning trees of the complete graph $K_d$ with vertex set $v_1,v_2,\ldots,v_d$.
\end{lemma}

\begin{theorem}\label{thm:multipartite}
Let $d > 1$ be fixed, and consider complete multipartite graphs $K_{n_1,n_2,\ldots,n_d}$ with $n = n_1 + n_2 + \cdots + n_d$ such that $\frac{n_i}{n}$ converges to a positive constant $c_i$ for each $i$ as $n \to \infty$. Take two independent spanning trees, both chosen uniformly at random. The number of common edges converges in distribution to a Poisson random variable with expected value
\[\sum_{1 \leq i < j \leq d} \frac{c_i c_j (2-c_i-c_j)^2}{(1-c_i)^2(1-c_j)^2}.\]
\end{theorem}

\begin{remark}
    In the special case $d=2$ (complete bipartite graphs), the expression for the expected value reduces to $\frac{1}{c_1c_2}$. In the special case $c_1 = c_2 = \cdots = c_d = \frac1d$ (regular multipartite graphs), it reduces to $\frac{2d}{d-1}$.
\end{remark}

\begin{proof}
In our setting, $n - n_i = \Omega(n)$ for every $i$, thus $\ga_j = \Omega(n)$ for every $j$. In the definition of $a_{pq}$, $n_{pj}n_{qj} = 0$ whenever the component $T_j$ of $F$ is a single vertex, since this vertex cannot be in both $V_p$ and $V_q$. Moreover, if the number of edges in $F$ is bounded by a constant $M$, then $\sum_{j=1}^k n_{pj}n_{qj} \leq M^2$ since at most $M$ vertices of $V_p$ and at most $M$ vertices of $V_q$ can lie in nontrivial components of $F$. Thus, for a bounded number of edges in $F$, we have
\[(d-1)a_{pq} = (d-1) \sum_{j=1}^k \frac{n_{pj}n_{qj}}{\ga_j} = O(n^{-1}).\]
This in turn means that the number of spanning trees of the complete multipartite graph $K_{n_1,n_2,\ldots,n_d}$ that contain $F$ is
\[\frac{\prod_{j=1}^k \ga_j}{(d-1)^{d-2} \prod_{i=1}^d (n-n_i)^2} \sum_{T \in \mathcal{T}(K_d)} \prod_{v_pv_q \in E(T)} (n-n_p)(n-n_q) \big( 1 - O(n^{-1}) \big).\]
This implies that the proportion of these spanning trees is \[\prod_{j=1}^k \ga_j \prod_{i=1}^d (n-n_i)^{-n_i} (1 - O(n^{-1})),\] since the product $\prod_{j=1}^k \ga_j$ is the only remaining part that depends on $F$, and this product evaluates to $\prod_{i=1}^d (n-n_i)^{n_i}$ for the forest that only consists of single vertices. This follows from the fact that if the component $T_j$ of $F$ is a single vertex in $V_x$, then $\ga_j = n-n_x$. On the other hand, 
if the component $T_j$ of $F$ is a single edge between partite sets $V_x$ and $V_y$, then we have
\[\ga_j = (n-n_x) + (n-n_y) = 2n-n_x-n_y.\]
Thus, if $F$ consists of disjoint edges $e_1,e_2,\ldots,e_h$, where $e_j$ connects $V_{x_j}$ and $V_{y_j}$, and otherwise single vertices, then the proportion of spanning trees containing $F$ is
\begin{equation}\label{eq:asymp_indep}
\prod_{j=1}^h \frac{2n-n_{x_j}-n_{y_j}}{(n-n_{x_j})(n-n_{y_j})}  \big( 1 - O(n^{-1}) \big).
\end{equation}
Here, the error term is uniform for every fixed $h$. This means that sets of disjoint edges are asymptotically independent. Moreover, even if the edges are not necessarily disjoint, it follows from Lemma~\ref{lem:lcy-formula} that the proportion is uniformly bounded by $O(n^{-h})$.

Now we can follow the same lines as in the proof of Proposition~\ref{prop:moments}: if we let $X_n$ be the number of common edges and let $Z_e$ be an indicator random variable for each edge again, then the formula
\[\Ex(X_n^r) = \sum_{s=1}^r \sum_{\{e_1,e_2,\ldots,e_s\}} r! [t^r] (e^t-1)^s \Ex(Z_{e_1}Z_{e_2}\cdots Z_{e_s})\]
remains true. Edge sets $\{e_1,e_2,\ldots,e_s\}$ that are not disjoint are negligible as their number is $O(n^{2s-1})$. Thus we can use asymptotic independence (as given by~\eqref{eq:asymp_indep}) to replace $\Ex(Z_{e_1}Z_{e_2}\cdots Z_{e_s})$ by $\Ex(Z_{e_1})\Ex(Z_{e_2})\cdots \Ex(Z_{e_s})$. This yields
\[\Ex(X_n^r) \sim \sum_{s=1}^r \frac{r! [t^r] (e^t-1)^s}{s!} \Big( \sum_e \Ex(Z_e) \Big)^s.\]
Since the number of edges between $V_i$ and $V_j$ is asymptotically equal to $c_i c_j n^2$, we have
\[\sum_e \Ex(Z_e) \sim \sum_{1 \leq i < j \leq d} c_i c_j n^2 \Big( \frac{2n-n_i-n_j}{(n-n_i)(n-n_j)} \Big)^2 \sim \sum_{1 \leq i < j \leq d} \frac{c_i c_j (2-c_i-c_j)^2}{(1-c_i)^2(1-c_j)^2}.\]
This shows that the moments of $X_n$ converge, completing the proof in the same way as for Theorem~\ref{thm:labelled1}.
\end{proof}

\section{General graphs}

Theorems~\ref{thm:labelled1} and~\ref{thm:multipartite} suggest the following question: under which general conditions does the number of common edges of two independent random spanning trees of a large graph $G$ approximately follow a Poisson distribution? This is clearly not the case for all graphs: for example, if $G$ is a tree with $n$ vertices, then the number of common edges is trivially $n-1$. If we consider the complete bipartite graph $K_{k,n-k}$ for fixed $k$, the number of common edges approximately follows a binomial distribution $\mathrm{Bin}(n-k,\frac1k)$ (and is thus asymptotically normally distributed).

One expects a Poisson distribution for denser graphs, but even then this is not necessarily the case: for example, if the graph has bridges, then these are always common edges. As another example, if we consider a graph consisting of two complete graphs $K_{n/2}$ connected by a fixed number $k$ of disjoint edges, then each of these edges belongs to a random spanning tree with a probability that is bounded below by $\frac1k$, resulting in a non-Poisson distribution once again.

\end{document}